 \documentclass[final,1p,times]{elsarticle}

\usepackage{amssymb}
 \usepackage{amsthm}
\usepackage{amsmath,amssymb,amsopn,amsfonts,mathrsfs,amsbsy,amscd}

\newcommand{\prs}{\langle\;,\;\rangle}
\newcommand{\too}{\longrightarrow}

\newcommand{\br}{[\;,\;]}

\newcommand{\vect}{{\mathrm{span}}}

\newcommand{\G}{{\mathfrak{g}}}

\newcommand{\aaa}{{\mathfrak{a}}}
\newcommand{\bbb}{{\mathfrak{b}}}
\newcommand{\g}{{\mathfrak{g}}}
\newcommand{\ad}{{\mathrm{ad}}}

\newcommand{\tr}{{\mathrm{tr}}}
\newcommand{\ric}{{\mathrm{ric}}}
\newcommand{\Ri}{{\mathrm{Ric}}}

\newcommand{\B}{{\cal B}}

\newcommand{\wi}{\widetilde}
\newcommand{\al}{\alpha}

\newcommand{\e}{\epsilon}

\newcommand{\la}{\lambda}

\font\bb=msbm10

\def\R{\hbox{\bb R}}

\newtheorem{theo}{Theorem}[section]
\newtheorem{pr}{Proposition}[section]
\newtheorem{Le}{Lemma}[section]

\newtheorem{rem}{Remark}

\begin{document}

\begin{frontmatter}


 \author{Oumaima Tibssirte}
 

\title{Einstein Lorentzian solvable unimodular Lie groups}

 \address[label2]{Universit\'e Priv\'ee de Marrakech\\
 Km 13,
 Route d’Amizmiz 42312 - Marrakech - Maroc\\e-mail: o.tibssirte@upm.ac.ma
 }



\begin{abstract}
The goal of this paper is to show that many key results found in the study of Einstein Lorentzian nilpotent Lie algebras can still hold in the more general settings of unimodular Lie algebras and (completely) solvable Lie algebras.
\end{abstract}

\begin{keyword} Einstein Lorentzian manifolds \sep  Nilpotent Lie groups  \sep Nilpotent Lie algebras \sep Unimodular Lie algebras \sep Solvable Lie algebras \sep Double extension process
\MSC 53C50 \sep \MSC 53D15 \sep \MSC 53B25


\end{keyword}

\end{frontmatter}






\section{Introduction}
\noindent Left-invariant pseudo-Riemannian Einstein metrics on nilpotent Lie groups is a research area that has known a significant progress, especially in the last decade. While the celebrated theorem of Milnor (see \cite{milnor}) has among its consequences that in the Riemannian setting no such metrics could exist, except when the underlying Lie group is abelian, the indefinite case is an entirely different story with a handful of examples available in the literature (see \cite{Dconti1}) and a systematic study of the Lorentzian case (see \cite{bou0}, \cite{Boucetta2} and \cite{guediri}). Although, for the time being, there is still no decisive theorem that settles the debate concerning the exact conditions for such metrics to exist, the abundance of results on this particular subject is more apparent now than ever before.\\\\
In contrast, there is little to no information when it comes to the investigation of left-invariant Einstein pseudo-Riemannian, indefinite metrics, on general Lie groups, and a line of study that could potentially lead to a similar development as in the nilpotent case is yet to be found, furethermore it is not known whether the results obtained for nilpotent Lie groups could hold under weaker conditions or admit a useful generalization, even solvable Lie groups which represent in some sense a natural extension of the nilpotent scene, remain a question mark.\\\\
The purpose of this paper is therefore to initiate such an inspection by expanding the context of the main theorems provided by the study of Einstein left-invariant Lorentzian metrics on nilpotent Lie groups to more general situations with an emphasis on the solvable case, and at the same time to shed light on the limitations of these results. The main reference for this paper is \cite{bou0}, and an analogy will be drawn between the different situations throughout the paper.\\\\
\textbf{Paper Outline.} In section 2, we introduce some general preliminaries on pseudo-Euclidean vector spaces and Lie algebras as well as the notations that will be used in the remainder of the paper, the proof of all the listed preliminary results can be found in either \cite{bou0} or \cite{Boucetta2}. We also recall the definition of a solvable Lie algebra and we give a brief account on the Killing form of a completely solvable lie algebra which shall be useful in the proofs of the main theorems. We start section 2 with a generalization of Theorem $7$ in \cite{Dconti1} to the case of an arbitrary Einstein pseudo-Euclidean, unimodular Lie algebra (see Proposition \ref{pr5}), the proof is based on Lemma \ref{pr4} and is similar in spirit to the one found in \cite[Proposition $3.7$]{bou0}. Next we generalize the Lorentzian version of \cite[Proposition $3.1$]{bou0} to the unimodular solvable setting and we observe that the proof can be further generalized to include any arbitray Einstein pseudo-Euclidean Lie algebra with non-degenerate center (see Proposition \ref{pr2}). We dedicate section 4 to establish a variant of the double extension process introduced by Medina-Revoy in \cite{medina} (also studied in the nilpotent setting in \cite{bou0}), we then give necessary and sufficient conditions on the parameters of double extensions that results in Einstein Lorentzian Lie algebras (see Proposition \ref{prdext1} and \ref{prdext2}), we then prove Theorems \ref{thdext1} and \ref{thdext2} which are broader versions of Theorem \cite[Theorem $4.1$]{bou0} and shows that essentially the same result holds more generally for completely solvable unimodular Lie algebras.
\section{Preliminaries}\label{section2}

\noindent A \emph{pseudo-Euclidean  vector space } is  a real vector space  of finite dimension $n$
endowed with  a
nondegenerate symmetric inner product  of signature $(q,n-q)=(-\ldots-,+\ldots+)$.  When
the
signature is $(0,n)$
(resp. $(1,n-1)$) the space is called \emph{Euclidean} (resp. \emph{Lorentzian}).\\

\noindent Let $(V,\prs)$ be a pseudo-Euclidean vector space of signature  $(q,n-q)$. A vector $u\in V$  is called \emph{spacelike} if $\langle u,u\rangle>0$, \emph{timelike} if $\langle u,u\rangle<0$ and \emph{isotropic} if  $\langle u,u\rangle=0$.
 A family
$(u_1,\ldots,u_s)$ of vectors in $V$ is called \emph{orthogonal}  if, for $i,j=1,\ldots,s$
and $i\not=j$, $\langle u_i,u_j\rangle=0$. An orthonormal basis of $V$ is an orthogonal basis $(e_1,\ldots,e_n)$ such that $\langle e_i,e_i\rangle=\pm1$.
A \emph{pseudo-Euclidean basis} of $V$ is a basis $(e_1,\bar e_2,\ldots,e_q,\bar
e_q,f_1,\ldots,f_{n-2q})$ for which the non vanishing products are 
$\langle \bar e_i, e_i\rangle=\langle f_j,f_j\rangle=1$, $i\in\{1,\ldots,q\}$ and $j\in\{1,\ldots,n-2q \}$.
 When $V$ is Lorentzian, we call such a basis
\emph{Lorentzian}.   Pseudo-Euclidean basis always exist.\\

\noindent For any endomorphism $F:V\too V$, we denote by $F^*:V\too V$ its adjoint with respect to $\prs$. We shall make use of the following Lemmas whose proofs can be found in \cite{bou0} :
\begin{Le}\label{le1} Let $(V,\prs)$ be a Lorentzian vector space,  $e$ an isotropic vector and $A$ a skew-symmetric endomorphism. Then $\langle Ae,Ae\rangle\geq0$. Moreover,  $\langle Ae,Ae\rangle=0$ if and only if $Ae=\al e$ with $\al\in\R$.\end{Le}
\begin{Le}\label{le} Let $(V,\prs)$ be a Lorentzian vector space, $e$ an isotropic vector and $A$ a skew-symmetric endomorphism such that $A(e)=0$. Then: 
	\begin{enumerate}
		\item $\tr(A^2)\leq0$,
		\item $\tr(A^2)=0$ if and only if for any $x\in e^\perp$, $A(x)=\la(x)e$ and in this case $\tr(A\circ B)=0$ for any skew-symmetric endomorphism satisfying $B(e)=0$.
	\end{enumerate}
	
\end{Le}

\noindent A Lie group $G$ together with a left invariant pseudo-Riemannian metric $g$ is called a 
\emph{pseudo-Riemannian Lie group}. The  metric $g$ 
defines a  symmetric nondegenerate inner
product $\prs$ on the Lie algebra $\G=T_eG$ of $G$, and conversely, any nondegenerate symmetric
inner product on $\G$
gives rise
to an unique  left invariant pseudo-Riemannian metric on $G$.\\

\noindent We will refer to a Lie
algebra endowed with a nondegenerate symmetric inner
product as a \emph{pseudo-Euclidean Lie algebra}. \\ Levi-Civita connection of $(G,g)$ defines a product $\mathrm{L}:\G\times\G\too\G$ called \emph{ Levi-Civita product}  given by  Koszul's
formula
\begin{eqnarray}\label{levicivita}2\langle
\mathrm{L}_uv,w\rangle&=&\langle[u,v],w\rangle+\langle[w,u],v\rangle+
\langle[w,v],u\rangle.\end{eqnarray}
For any $u,v\in\G$, $\mathrm{L}_{u}:\G\too\G$ is skew-symmetric and $[u,v]=\mathrm{L}_{u}v-\mathrm{L}_{v}u$.
The curvature on $\G$ is given by
$$
\label{curvature}K(u,v)=\mathrm{L}_{[u,v]}-[\mathrm{L}_{u},\mathrm{L}_{v}].
$$
The Ricci curvature $\mathrm{ric}:\G\times\G\too\R$ and its Ricci operator $\Ri:\G\too\G$ are defined by $\langle \Ri (u),v\rangle=\mathrm{ric}(u,v)=\mathrm{tr}\left(w\too
K(u,w)v\right)$. A pseudo-Euclidean Lie algebra is called \emph{flat} (resp. \emph{Ricci-flat}) if $K=0$
(resp. $\ric=0$). It is called Einstein if there exists a constant $\la\in\R$ such that $\Ri=\la\mathrm{Id}_\G$.
 For any $u\in\G$, put $\mathrm{R}_u=\mathrm{L}_u-\ad_u$.
It is well-known that $\mathrm{ric}$ is  given by
\begin{equation}\label{ric1}
\ric(u,v)=-\tr(\mathrm{R}_u\circ
\mathrm{R}_v)-\frac12\left({\langle}\ad_Hu,v{\rangle}+{\langle}\ad_Hv,u{\rangle}\right),
\end{equation} where $H$ is the vector given by $\langle H,u\rangle=\mathrm{tr(\ad_u)}$.
Let us transform this formula to get some
useful formulas. To do so, we consider  $\ad:\G\too\mathrm{End}(\G)$ and $J:\G\too\mathrm{so}(\G,\prs)$  the adjoint representation and the endomorphism given by $J_u(v)=\ad_v^*u$. It is clear that $J_u$ is skew-symmetric, $\ker\ad=Z(\G)$ and $\ker J=[\G,\G]^\perp$.
One can deduce easily from (\ref{levicivita}) that
$$\mathrm{R}_u=-\frac12\left(\ad_u+\ad_u^*\right)-\frac12J_u.$$ 
Thus
\begin{equation}\label{ricci1}
\ric(u,v)=-\frac12\tr(\ad_u\circ\ad_v)-\frac12\tr(\ad_u\circ
\ad_v^*)-\frac14\tr(J_u\circ J_v)-\frac12{\langle}\ad_Hu,v{\rangle}-
\frac12{\langle}\ad_Hv,u{\rangle}.\end{equation}
Define now the auto-adjoint endomorphisms  $\widehat{B}$, ${\mathcal J}_1$
and
${\mathcal J}_2$ by
\begin{eqnarray*}
	\langle\widehat{B} u,v\rangle=\tr(\ad_u\circ\ad_v),\;
	\langle{\mathcal J}_1 u,v\rangle=\tr(\ad_u\circ\ad_v^*),\;
	\langle{\mathcal J}_2 u,v\rangle=-\tr(J_u\circ J_v)=\tr(J_u\circ J_v^*).\end{eqnarray*}
Thus (\ref{ricci1}) is equivalent to
\begin{equation}\label{ricci3}\mathrm{Ric}=-\frac12\left(\widehat{B}+{\mathcal
	J}_1\right)+\frac14{\mathcal J}_2-\frac12(\ad_H+\ad_H^*).
\end{equation}
\noindent When $\G$ is a unimodular Lie algebra then $H=0$ and hence
\eqref{ricci3} becomes
\begin{equation}\label{riccinilpotent}\mathrm{Ric}=-\frac12(\widehat{B}+{\mathcal
		J}_1)+\frac14{\mathcal J}_2,\end{equation}
Since we will deal with nilpotent Lie algebras, only ${\mathcal J}_1$ and ${\mathcal J}_2$ will be relevant so we are going to express them in an useful way. This is based on the notion of structure endomorphisms we introduce now.\\

Let
$(e_1,\ldots,e_p)$ be  a basis of
$\G$. Then, for any $u,v\in\G$, the Lie bracket can be written
\begin{equation}\label{bracket}[u,v]=\sum_{i=1}^p\langle S_iu,v\rangle
e_i,\end{equation}where $S_i:\G\too\G$ are skew-symmetric
endomorphisms with respect to $\prs$. 
 The family $(S_1,\ldots,S_p)$ will be called
\emph{structure endomorphisms} associated to $(e_1,\ldots,e_p)$. One can see  easily from the definition of $J$ and \eqref{bracket} that for
any
$u\in\G$,
\begin{equation}\label{J}J_u=\sum_{i=1}^p\langle u,e_i\rangle S_i\end{equation}
The structure endomorphisms $(K_1,\ldots,K_p)$ associated to a new basis
$(f_1,\ldots,f_p)$  are given by
$$\label{changementofbasis}
K_j=\sum_{i=1}^pp^{ji}S_i,\quad j=1,\ldots,p,$$where $(p^{ij})_{1\leq i,j\leq
	p}$ is  the passage matrix from $(f_1,\ldots,f_p)$ to $(e_1,\ldots,e_p)$.
The following proposition will be useful later, its proof is a matter of simple computation and can be found in \cite{bou0}.
\begin{pr}\label{jjj}Let $(\G,\prs)$ be a pseudo-Euclidean   Lie algebra,
	$(e_1,\ldots,e_p)$  a basis of
	$\G$ and $(S_1,\ldots,S_p)$ the corresponding structure endomorphisms. Then 
	\begin{equation}\label{invariant}
	{\mathcal J}_1=-\sum_{i,j=1}^p\langle e_i,e_j\rangle S_i\circ S_j\quad
	\mbox{and}\quad {\mathcal J}_2u=-\sum_{i,j=1}^p\langle e_i,u\rangle{\tr}(S_i\circ
	S_j)e_j.\end{equation}In particular, $\tr{\mathcal J}_1=\tr{\mathcal J}_2$.
\end{pr}
\noindent Let $\g$ be a real Lie algebra and denote $\mathcal{D}^0(\g):=\g$ and for all $k\geq 1$, $\mathcal{D}^{k}(\g):=[\mathcal{D}^{k-1}(\g),\mathcal{D}^{k-1}(\g)]$.Recall that $\g$ is said to be solvable if it satisfies $\mathcal{D}^k(\g)=0$ for some $k\geq 1$. We say that the Lie algebra $\g$ is completely solvable if it has a chain of ideals $$0=I_0\subset I_1\subset\dots\subset I_n=\g,$$ such that $\dim(I_i)=i$ for all $0\leq i\leq n$. In this case, $\g$ admits a basis $\{e_1,\dots,e_n\}$ in which every endomorphism $\ad_u$ with $u\in\g$ is represented by an upper triangular matrix $M_u$, in particular the Killing form $B$ of $\g$ satisfies :
$$ B(u,u)=\tr(\ad_u^2)=\tr(M_u^2)=\sum_{i=1}^n (M_u)_{ii}^2\geq 0,$$
i.e the Killing form of a completely solvable real Lie algebra is positive semi-definite.Finally note that by Engel's Theorem any nilpotent Lie algebra is completely solvable.
\section{Results on solvable and unimodular Lorentzian Einstein Lie algebras}
\begin{Le}\label{pr4} Let $(\G,\prs)$ be a pseudo-Euclidean Lie algebra and let $Q$ denote the symmetric endomorphism $Q=-\frac12\mathcal{J}_1+\frac14\mathcal{J}_2$. Then for any orthonormal basis $(e_1,\ldots,e_p)$ of $\G$ and any endomorphism $E$ of $\G$, we have
	\begin{equation}\label{formula} \tr(QE)=\frac14\sum_{i,j}\e_i\e_j\langle E([e_i,e_j])-[E(e_i),e_j]-[e_i,E(e_j)],[e_i,e_j]\rangle, \end{equation}where $\langle e_i,e_i\rangle=\e_i$.
	
\end{Le}
\begin{proof} We denote by $(J_1,\ldots,J_p)$ the structures endomorphisms associated to $(e_1,\ldots,e_p)$. By combining \eqref{invariant} and \eqref{J} we get
	\[ QE(u)=\frac12\sum_{i=1}^p\e_i J_{e_i}^2E(u)-\frac14\sum_{i,j=1}^p\e_i\e_j\langle e_i,E(u)\rangle{\tr}(J_{e_i}\circ
	J_{e_j})e_j. \]
	Let compute:
	\begin{eqnarray*}
		\tr(QE)&=&\sum_{j=1}^p\e_j\langle QE(e_j),e_j\rangle\\
		&=&-\frac12\sum_{i,j=1}^p\e_i\e_j\langle J_{e_i}E(e_j),J_{e_i}(e_j)\rangle -\frac14\sum_{i,j=1}^p\e_i\e_j\langle e_i,E(e_j)\rangle{\tr}(J_{e_i}\circ
		J_{e_j})\\
		&=&-\frac12\sum_{i,j=1}^p\e_j\e_i\langle e_i,[E(e_j),J_{e_i}(e_j)]\rangle+\frac14\sum_{i,j,l=1}^p\e_i\e_l\e_j\langle e_i,E(e_j)\rangle\langle J_{e_j}e_l,J_{e_i}e_l\rangle\\
		&=&-\frac12\sum_{i,j,l=1}^p\e_j\e_i\e_l\langle J_{e_i}(e_j),e_l\rangle \langle e_i,[E(e_j),e_l]\rangle+\frac14\sum_{j,l=1}^p\e_l\e_j\langle J_{e_j}e_l,J_{E(e_j)}e_l\rangle\\
		&=&-\frac12\sum_{i,j,l=1}^p\e_j\e_i\e_l\langle {e_i},[e_j,e_l]\rangle \langle e_i,[E(e_j),e_l]\rangle+\frac14\sum_{i,j,l=1}^p\e_l\e_j\e_i\langle J_{e_j}e_l,e_i\rangle \langle e_i,J_{E(e_j)}e_l\rangle\\
		&=&-\frac12\sum_{j,l=1}^p\e_j\e_l \langle [e_j,e_l],[E(e_j),e_l]\rangle+\frac14\sum_{i,j,l=1}^p\e_l\e_j\e_i\langle {e_j},[e_l,e_i]\rangle \langle [e_l,e_i],{E(e_j)}\rangle\\
		&=&-\frac12\sum_{j,l=1}^p\e_j\e_l \langle [e_j,e_l],[E(e_j),e_l]\rangle+\frac14\sum_{i,l=1}^p\e_l\e_i \langle [e_l,e_i],{E([e_l,e_i])}\rangle\\
		&=&	-\frac14\sum_{j,l=1}^p\e_j\e_l \langle [e_j,e_l],[E(e_j),e_l]\rangle-\frac14\sum_{j,l=1}^p\e_j\e_l \langle [e_j,e_l],[e_j,E(e_l)]\rangle\\&&+\frac14\sum_{i,l=1}^p\e_l\e_i \langle [e_l,e_i],{E([e_l,e_i])}\rangle,
	\end{eqnarray*}and the formula follows.
\end{proof}
\begin{pr}\label{pr5} Let $(\G,\prs)$ be a pseudo-Euclidean unimodular Lie algebra having a derivation with non null trace. Then $(\G,\prs)$ is Einstein if and only if it is Ricci flat.
\end{pr}
\begin{proof}
	Denote $Q:=-\frac{1}{2}\mathcal{J}_1+\frac{1}{4}\mathcal{J}_2$. Since $\g$ is unimodular, equation \eqref{riccinilpotent} gives that :
	$$ \mathrm{Ric}=Q-\frac{1}{2}\widehat{B}.$$
	Next, recall that the Killing form of $\g$ satisfies $B(\tau(u),\tau(v))=B(u,v)$ for any $\tau\in\mathrm{Aut}(\g)$ and any $u,v\in\g$, therefore $B(Du,v)=-B(u,Dv)$ for any $D\in\mathrm{Der}(\g)$, in particular $B(Du,u)=0$. Fix an orthonormal basis $\{e_1,\dots,e_n\}$ of $(\g,\prs)$ and set $\epsilon_i:=\langle e_i,e_i\rangle$, then for any $D\in\mathrm{Der}(\g)$ :
	$$\tr(\widehat{B}\circ D)=\sum_{i=1}^n \epsilon_i \langle \widehat{B}(D(e_i)),e_i\rangle=\sum_{i=1}^n \epsilon_i B(D(e_i),e_i)=0,$$
	on the other hand Lemma \ref{pr4} gives that $\tr(Q\circ D)=0$, this shows that $\tr(\mathrm{Ric}\circ D)=0$ and so if $\g$ is $\lambda$-Einstein we get that $\lambda\tr(D)=0$, which proves the claim.
\end{proof}
\begin{pr}\label{pr2} Let $(\G,\prs)$ be an Einstein Lorentzian solvable unimodular Lie algebra. If $Z(\G)$ is nondegenerate then it is either Euclidean or $(\G,\prs)$ is flat, and $\G$ can be written as $\G=\bbb\oplus \aaa$ where $\bbb$ is an abelian ideal and $\aaa$ is an abelian subalgebra of $\G$, moreover the Levi-Civita product of $\G$ has the form
	$$\mathrm{L}_u=\left\{\begin{array}{l c r}\ad_u& &u\in\aaa\\
	0,& &\text{otherwise}\end{array}\right.$$
	
\end{pr}

\begin{proof} Suppose that $Z(\G)$ is nondegenerate Lorentzian and choose a vector $e\in Z(\G)$ such that $\langle e,e\rangle=-1$. We have $\G=\R e\oplus\G_0$, where $\G_0=e^\perp$. 
	For any $u,v\in\G_0$, put $[u,v]=\langle Ku,v\rangle e+[u,v]_0$, where $K:\G_0\too\G_0$ and $[u,v]_0\in\G_0$.
	It is obvious that $(\G_0,\prs)$ is a Euclidean solvable Lie algebra, moreover $\g_0$ is unimodular. We claim that if $(\G,\prs)$ is Einstein i.e $\mathrm{Ric}_\g=\lambda\mathrm{Id}_\g$ then $\lambda=\frac{1}{4}\tr(K^2)$ and
	$$\Ri_{\prs_0}=\frac14\tr(K^2)\mathrm{Id}_{\G_0}+\frac12K^2.$$
	This implies that the Ricci curvature of $(\G_0,\prs)$ is negative. But since $\g_0$ is a unimodular solvable Euclidean Lie algebra, \cite[Corollary $3.3$]{dotti} gives that its Ricci curvature cannot be negative definite, therefore $\lambda=0$ and so $K=0$ which implies that $\g_0$ is an ideal since $[\g,\g]=[\g_0,\g_0]$. Moreover \cite[Theorem 
	$3.1$]{dotti} shows that $\g_0=\bbb_0\oplus\aaa$ where $\aaa$ is an abelian subalgebra of $\g_0$ and $\bbb_0$ is an abelian ideal of $\g_0$, and if $\mathrm{L}^0$ denotes the Levi-Civita product of $\g_0$ then for any $u\in\g_0$
	$$\mathrm{L}^0_u=\left\{\begin{array}{l c r}\ad_u& &u\in\aaa\\
	0,& &u\in\bbb_0\end{array}\right.$$
	Thus if we set $\bbb:=\R e\oplus[\g_0,\g_0]$ we get the desired result. Let prove our claim. Indeed, we choose an orthonormal basis $(e_1,\ldots,e_d)$ of $\G_0$ and we denote by $(\wi K,\wi S_1,\ldots,\wi S_d)$ the structure endomorphisms associated to the basis $(e,e_1,\dots,e_d)$ of $\g$. Note that all these endomorphisms vanishes at $e$ and hence leave $\G_0$ invariant. We have ${\wi K}_{|\G_0}=K$ and if $S_i={(\wi S_i)}_{|\G_0}$ then $(S_1,\ldots,S_d)$ are the structure endomorphisms of $\G_0$ corresponding to $(e_1,\ldots,e_d)$. If $\G$ is Einstein, according to  \eqref{invariant},
	\[ -\frac12 \widehat{B} -\frac12\wi K^2+\frac12\sum_{i=1}^d\wi S_i^2+\frac14\mathcal{J}_2=\la\mathrm{Id}_\G.\eqno(*) \] 
	If we evaluate this expression at $e$ using the expression of $\mathcal{J}_2$ given in  \eqref{invariant}, we get
	\[ \frac14\tr(\wi K^2)e-\frac14\sum_{i=1}^d\tr(\wi K\circ\wi S_i )e_i=\la e. \]	
	This is equivalent to $\la=\frac14\tr(\wi K^2)=\frac14\tr( K^2)$ and $\tr(\wi K\circ\wi S_i )=0$ for $i\in\{1,\ldots,d \}$, we also observe that the Killing form of $(\g_0,[\;,\;]_0)$ is equal to the restriction to $\g_0$ of the Killing form of $(\g,[\;,\;])$. By taking the restriction of $(*)$ to $\G_0$ we finish the proof of our claim.	
\end{proof}
\begin{rem}
The previous result can be generalized to arbitrary Einstein pseudo-Euclidean, solvable and unimodular Lie algebra, the proof is essentially the same.
\end{rem}
\section{The Double Extension Theorem For Solvable Unimodular Lorentzian Einstein Lie algebras}
Let $(\g_0,\br_0,\prs_0)$ be a Euclidean Lie algebra and set $\g:=\R e\oplus\g_0\oplus\R\bar{e}$. Let $\prs$ be the Lorentzian metric on $\g$ given by $\langle e,e\rangle=\langle\bar{e},\bar{e}\rangle=0$, $\langle e,\bar{e}\rangle=1$ and $\prs_{|\g_0\times\g_0}:=\prs_0$ then define on $\g$ the bracket $\br$ by the expression :
$$[\bar{e},e]=\mu e,\;\;\;[\bar{e},u]=D(u)+\langle b,u\rangle_0 e\;\;\;\text{and}\;\;\;[u,v]=[u,v]_0+\langle K(u),v\rangle_0\;\;\;\text{for any}\;u,v\in\g_0,$$
such that $\mu\in\R$, $b\in\g_0$ and $K,D:\g_0\longrightarrow\g_0$ are two endomorphisms with $K$ skew-symmetric. We say that the triple $(\g,\br,\prs)$ is obtained from $(\g_0,\br_0,\prs_0)$ by a double extension process with parameters $(K,D,\mu,b)$.
\begin{pr}
	\label{prdext1}
	Let $(\g_0,\br_0,\prs_0)$ be a Euclidean Lie algebra and consider the triple $(\g,\br,\prs)$ obtained from $(\g_0,\br_0,\prs_0)$ by a double extension process with parameters $(K,D,\mu,b)$. Denote $\ad^0$ the adjoint representation of $\g_0$.
	\begin{enumerate}
		\item Let $\omega:\g_0\times\g_0\longrightarrow\R$ be the $2$-form given by $\omega(u,v):=\langle K(u),v\rangle_0$, then $(\g,\br)$ is a Lie algebra if and only $D\in\mathrm{Der}(\g_0)$, $\omega\in\mathrm{Z}^2(\g_0,\R)$ i.e a $2$-cocycle and 
		\begin{equation}
		\label{dext0}
		KD+D^*K=\mu K+J_b^0,
		\end{equation}
		where $J_u^0(v):=(\ad_v^0)^*(u)$ for all $u,v\in\g_0$.
		\item Assume $\g$ is a Lie algebra, let $H\in\g$ and $H^0\in\g_0$ be the mean curvature vectors of $\g$ and $\g_0$ respectively i.e for any $u\in\g$ and $v\in\g_0$, $\tr(\ad_u)=\langle H,u\rangle$ and $\tr(\ad_v^0)=\langle H^0,v\rangle$. Then $$H=(\mu+\tr(D))e+H^0.$$ In particular $\g$ is unimodular if and only if $\g_0$ is unimodular and $\tr(D)=-\mu$.
	\end{enumerate}
\end{pr}
\begin{proof}
	$1.$ Let $u,v,w\in\g_0$, then it is clear that $[[u,v],w]=[[u,v]_0,w]_0+\langle K([u,v]_0),w\rangle$ hence
	$$\small\oint[[u,v],w]=\underset{=0}{\underbrace{\oint[[u,v]_0,w]_0}}+\oint \langle K([u,v]_0),w\rangle=-d\omega(u,v,w).$$
	Next, a straightforward computation shows that $[\bar{e},[u,v]]=D([u,v]_0)+\langle b,[u,v]_0\rangle e+\mu\langle K(u),v\rangle e$ and $[v,[\bar{e},u]]=[v,D(u)]_0+\langle K(v),D(u)\rangle e$, thus
	\begin{align*}
	\oint[\bar{e},[u,v]]=&[\bar{e},[u,v]]+[v,[\bar{e},u]]+[u,[v,\bar{e}]]\\
	=& D([u,v]_0)+\langle b,[u,v]_0\rangle e+\mu\langle K(u),v\rangle e+[v,D(u)]_0+\langle K(v),D(u)\rangle e-[u,D(v)]_0-\langle K(u),D(v)\rangle e\\
	=&D([u,v]_0)-[D(u),v]_0-[u,D(v)]_0+\langle (\ad_u^0)^*(b),v\rangle e+\langle \mu K(u),v\rangle e-\langle KD(u),v\rangle e-\langle D^*K(u),v\rangle e\\
	=&D([u,v]_0)-[D(u),v]_0-[u,D(v)]_0+\langle(J_b^0+\mu K-KD-D^*K)(u),v\rangle e.
	\end{align*}
	Now $(\g,\br)$ is a Lie algebra if and only if it satisfies Jacobi identity, this proves the claim in view of the previous computation.\\\\
	$2.$ First notice that since $\ad_e(\g)\subset \R e$ then $\langle H,e\rangle=\tr(\ad_e)=0$. On the other hand we have that $\ad_{\bar{e}}(e)=\mu e$ and $\ad_{\bar{e}}(u)=D(u)+\langle b,u\rangle e$ so that $\ad_u(\bar{e})\subset (\R e)^\perp$ and $\ad_u(v)=\ad_u^0(v)+\langle K(u),v\rangle$, for any $u,v\in\g_0$. Therefore
	$$\langle H,\bar{e}\rangle=\tr(ad_{\bar{e}})=\mu+\tr(D),\;\;\;\langle H,u\rangle=\tr(\ad_u)=\tr(\ad_u^0)=\langle H^0,u\rangle,$$
	and so we conclude that $H=(\mu+\tr(D))e+H^0$.
\end{proof}
\begin{pr}\label{prdext2}
	Let $\g_0$ be a Euclidean Lie algebra and $\g$ the Lie algebra obtained from $\g_0$ by a double extension process with parameters $(K,D,\mu,b)$. Then $\g$ is Einstein if and only if it Ricci-flat, $\g_0$ is Ricci-flat and for all $u\in\g_0$, the following equations are satisfied :
	\begin{align}
	4\tr(\ad_b^0)+4\mu\tr(D)-2\tr(D^2)-2\tr(DD^*)-\tr(K^2)&=0\label{dext1}\\
	\tr(J_u^0\circ K)+2\tr((D+D^*)\circ\ad_u^0)+2\tr(\ad_{D^*(u)})-2\tr(\ad_{K(u)})&=0.\label{dext2}
	\end{align}
\end{pr}
\begin{proof} Fix an orthonormal basis $(f_1,\dots,f_n)$ of $(\g_0,\prs_0)$ and denote by $\mathrm{Ric}$ and $\mathrm{Ric}_0$ the Ricci operators of $\g$ and $\g_0$ respectively and $\ad^0$ the adjoint representation of $\g_0$. According to \eqref{ricci3}, we can write :
	$$\mathrm{Ric}=-\frac12(\mathcal{J}_1+\widehat{B})-\frac12(\ad_H+\ad_H^*)+\frac14\mathcal{J}_2$$
	$$\mathrm{Ric_0}=-\frac12(\mathcal{J}_1^0+\widehat{B}^0)-\frac12(\ad^0_{H^0}+(\ad^0_{H^0})^*)+\frac14\mathcal{J}_2^0.$$
	Since $\ad_e(\g)\subset\R e$ and $\ad_e(e)=0$ then it follows that $\widehat{B}(e)=0$. Next, for any $u,v\in \g_0$, it is easy to check that :
	$$\begin{cases}\ad_{\bar{e}}^2(e)=\mu^2 e,\;\ad_{\bar{e}}(\bar{e})=0,\;{\ad_{\bar{e}}^2}_{|\g_0}=D^2+\langle D^*b+\mu b,\cdot\;\rangle e,\\
	\ad_{\bar{e}}\circ\ad_u(e)=0,\;\ad_{\bar{e}}\circ\ad_u(\bar{e})\in(\R e)^\perp,\;{\ad_{\bar{e}}\circ\ad_u}_{|\g_0}=D\circ\ad_u^0+(\langle b,\ad_u^0(\cdot)\rangle+\langle\mu K(u),\cdot\rangle e,\\
	\ad_u\circ\ad_v(e)=0,\;\ad_u\circ\ad_v(\bar{e})\in(\R e)^\perp,\;(\ad_u\circ\ad_v)_{|\g_0}=\ad_u^0\circ\ad_v^0+\langle K(u),\ad_v^0(\cdot)\rangle e.
	\end{cases}$$
	This shows that $\langle\widehat{B}(\bar{e}),\bar{e}\rangle=\mu^2+\tr(D^2)$, $\langle\widehat{B}(\bar{e}),u\rangle=\tr(D\circ\ad_u^0)$ and $\langle\widehat{B}(u),v\rangle=\langle\widehat{B}^0(u),v\rangle$, this leads to :
	\begin{equation}
	\label{doubleext1}
	\widehat{B}(u)=\widehat{B}^0(u)+\tr(D\circ\ad_u^0)e,\;\;\;\widehat{B}(\bar{e})=(\mu^2+\tr(D^2))e+\sum_{i=1}^n\tr(D\circ\ad_{f_i}^0)f_i.
	\end{equation}
	We have that $H=H^0+(\mu+\tr(D))e$ and therefore $\ad_H=\ad_{H^0}+(\mu+\tr(D))\ad_e$, this show that $\ad_H(e)=0$ and since $[\g,\g]\subset (\R e)^\perp$ i.e $e\in [\g,\g]^\perp$ then $\ad_H^*(e)=0$. Furthermore for any $u,v\in \g_0$, one can easily check that :
	$$\begin{cases}
	\langle \ad_H(u),v\rangle=\langle\ad_{H^0}^0(u),v\rangle,\;\langle\ad_H^*(u),v\rangle=\langle(\ad_{H^0}^0)^*(u),v\rangle,\\
	\langle\ad_H(u),\bar{e}\rangle=\langle K(H^0),u\rangle,\;\langle\ad_H^*(u),\bar{e}\rangle=-\langle D(H^0),u\rangle,\\
	\langle\ad_H(\bar{e}),\bar{e}\rangle=\langle\ad_H^*(\bar{e}),\bar{e}\rangle=-(\langle b,H^0\rangle+\mu(\mu+\tr(D))).
	\end{cases}$$
	In summary this gives that :
	\begin{align}
	(\ad_H+\ad_H^*)(e)=&0\label{doubleext2.1}\\
	(\ad_H+\ad_H^*)(\bar{e})=&-2(\tr(\ad_b^0)+\mu(\mu+\tr(D)))e+(K(H^0)-D(H^0))\label{doubleext2.2}\\(\ad_H+\ad_H^*)(u)=&(\ad_{H^0}^0+(\ad_{H^0}^0)^*)(u)-\tr(\ad_{K(u)+D^*(u)}^0)e.\label{doubleext2.3}
	\end{align}
	Now let $(A,S_1,\dots,S_n)$ be structure endomorphisms of $(\g,\br)$ corresponding to the family $(e,\bar{e},f_1,\dots,f_n)$ and $(\tilde{S}_1,\dots,\tilde{S}_n)$ of $(\g_0,\br_0)$ corresponding to $(f_1,\dots,f_n)$ then from \eqref{??} we get that for any $u\in\g$ and any $v\in\g_0$ :
	$$\mathcal{J}_1=-\sum_{i=1}^n S_i^2\;\;\;\mathcal{J}_1^0=-\sum_{i=1}^n\tilde{S}_i^2$$
	$$\mathcal{J}_2(u)=-\langle e,u\rangle\tr(A^2)e-\langle e,u\rangle\sum_{i=1}^n\tr(S_i\circ A)f_i-\sum_{i=1}^n\langle f_i,u\rangle\tr(S_i\circ A)e-\sum_{i,j=1}^n\langle f_i,u\rangle\tr(S_i\circ S_j)f_j$$
	$$\mathcal{J}_2^0(v)=-\sum_{i,j=1}^n\langle f_i,v\rangle\tr(\tilde{S}_i\circ \tilde{S}_j)f_j.$$
	Since $\langle S_i(\tilde{u}),\tilde{v}\rangle=\langle [\tilde{u},\tilde{v}],f_i\rangle$, $\langle A(\tilde{u}),\tilde{v}\rangle=\langle\tilde{u},\tilde{v},\bar{e}\rangle$ and $\langle \tilde{S}_i(u),v\rangle=\langle[u,v]_0,f_i\rangle$ for all $\tilde{u},\tilde{v}\in\g$ and all $u,v\in\g_0$, one can easily see that
	
	$$\begin{cases}
	S_i(e)=0,\;\langle S_i(\bar{e}),u\rangle=D^*(f_i),\;\langle S_i(u),v\rangle=\langle \tilde{S}_i(u),v\rangle\\
	\langle A(e),\bar{e}\rangle=-\mu,\;\langle A(e),u\rangle=0,\;\langle A(\bar{e}),u\rangle=\langle b,u\rangle,\;\langle A(u),v\rangle=\langle K(u),v\rangle.\\
	\end{cases}$$
	This shows that $S_i(\bar{e})=D^*(f_i)$, $S_i(u)=\tilde{S}_i(u)-\langle D^*(f_i),u\rangle e$, $A(e)=-\mu e$, $A(\bar{e})=-\mu\bar{e}+b$ and $A(u)=\langle b,u\rangle e+K(u)$, and so we obtain that :
	$$\begin{cases}
	S_i\circ S_j(\bar{e})=\tilde{S}_i(D^*(f_j))-\langle D^*(f_i),D^*(f_j)\rangle,\\
	S_i\circ S_j(u)=\tilde{S}_i\circ\tilde{S}_j(u)-\langle (D\circ\tilde{S}_j)(u),f_i\rangle e,\\
	A^(e)=\mu^2 e,\;A^2(\bar{e})=\mu^2\bar{e}-\langle b,b\rangle e+K(b)-\mu b,\\
	A^2(u)=\langle K(b)+\mu b,u\rangle e+K^2(u),\\
	A\circ S_i(e)=0,\;A\circ S_i(\bar{e})=K(D^*f_i)-\langle Db,f_i\rangle e,\\
	A\circ S_i(u)=\langle \tilde{S}_i(b)-\mu D^*(f_i),u\rangle e+ K(\tilde{S}_i(u)).
	\end{cases}$$
	Consequently we deduce that :
	\begin{align}
	\mathcal{J}_1(\bar{e})=&-\sum_{i=1}^n(\tilde{S}_i\circ D^*)(f_i)+\tr(DD^*)e\label{doubleext3.1}\\
	\mathcal{J}_1(u)=&\mathcal{J}_1^0(u)+\sum_{i=1}^n\langle D\circ\tilde{S}_i(u),f_i\rangle e=\mathcal{J}_1^0(u)+\tr(\ad_u^0\circ D^*) e.\label{doubleext3.2}
	\end{align}
	The previous computation also gives that $\tr(S_i\circ S_j)=\tr(\tilde{S}_i\circ\tilde{S}_j)$, $\tr(A\circ S_i)=\tr(K\circ\tilde{S}_i)$ and $\tr(A^2)=2\mu^2+\tr(K^2)$ which then leads to :
	\begin{align}
	\mathcal{J}_2(\bar{e})=&-(2\mu^2+\tr(K^2))e-\sum_{i=1}^n\tr(K\circ\tilde{S}_i)f_i\label{doubleext4.1}\\
	\mathcal{J}_2(u)=&\mathcal{J}_2^0(u)-\tr(J_u^0\circ K)e.\label{doubleext4.2}
	\end{align}
	Finally if we combine equations \eqref{doubleext1} $\dots$ \eqref{doubleext4.2} then we conclude that for any $u,v\in\g_0$ :
	\begin{align*}
	\mathrm{Ric}(e)=&0,\;\;\;\langle\mathrm{Ric}(u),v\rangle=\langle\mathrm{Ric}_0(u),v\rangle,\\
	\langle\mathrm{Ric}(u),\bar{e}\rangle=&-\frac14(\tr(J_u^0\circ K)+2\tr((D+D^*)\circ\ad_u^0)+2\tr(\ad_{D^*(u)})-2\tr(\ad_{K(u)})),\\
	\langle\mathrm{Ric}(\bar{e}),\bar{e}\rangle=&\frac14(4\tr(\ad_b^0)+4\mu\tr(D)-2\tr(D^2)-2\tr(DD^*)-\tr(K^2)).
	\end{align*}
	Thus $(\g,\br,\prs)$ is Einstein if and only if it is Ricci-flat, $(\g_0,\br_0,\prs_0)$ is Ricci-flat and equations \eqref{dext1} and \eqref{dext2} are satisfied.
\end{proof}
\begin{theo}\label{thdext1}
	Let $\g$ be a completely solvable unimodular, Lorentzian Lie algebra such that $[\g,\g]$ is degenerate. Then $\g$ is Einstein if and only if it is obtained by a double extension process from an abelian Lie algebra with admissible parameters $(K,D,-\tr(D),b)$. In particular $\g$ is Ricci-flat.
\end{theo}
\begin{proof}
	Let $e\neq 0$ be a generator of $[\g,\g]\cap[\g,\g]^\perp$ and choose $\overline{e}\in\g$, $f_1,\dots,f_d\in[\g,\g]$ and $g_1,\dots,g_s\in[\g,\g]^\perp$ such that $(e,\overline{e},f_1,\dots,f_d,g_1,\dots,g_s)$ is a Lorentzian basis of $\g$. Let $(K,S_1,\dots,S_d)$ be the structure endomorphisms of $\g$ corresponding to the basis $(e,f_1,\dots,f_d)$ of $[\g,\g]$. From \eqref{invariant} it is clear that for any $u\in\g$ :
	$$-\mathcal{J}_2(u)=\langle e,u\rangle\tr(K^2)e+\sum_{i=1}^d\langle e,u\rangle\tr(K\circ S_i)f_i+\sum_{i=1}^d\langle f_i,u\rangle\tr(S_i\circ K)e+\sum_{i,j=1}^d\langle f_i,u\rangle\tr(S_i\circ S_j)f_j.$$
	Since $e\in[\g,\g]^\perp$ then by Proposition \ref{jjj}, $\mathcal{J}_2(e)=0$ and since $e\in[\g,\g]$ and $\g$ is solvable then $\hat{\mathcal{B}}(e)=0$. This shows that $\mathcal{J}_1(e)=-2\lambda e$. On the other hand it is clear by \eqref{invariant} that :
	$$ \mathcal{J}_1=-\sum_{i=1}^d S_i^2,$$
	therefore using that $\langle S_i(e),e\rangle=0$, we deduce that :
	$$ 0=\langle -2\lambda e,e\rangle=\langle\mathcal{J}_1(e),e\rangle=\sum_{i,j=1}^d\langle S_i(e),f_j\rangle^2+\sum_{i=1}^d\sum_{j=1}^s\langle S_i(e),g_j\rangle^2.$$
	This shows that $S_i(e)\in\R e$ for $i=1,\dots,d$ i.e $S_i(e)=\alpha_i e$ for some $\alpha_i\in\R$. Consequently :
	$$\lambda=\langle\lambda e,\overline{e}\rangle=-\dfrac{1}{2}\langle\mathcal{J}_1(e),\overline{e}\rangle=\dfrac{1}{2}\sum_{i=1}^d\alpha_i^2\geq 0.$$
	It is clear from \eqref{invariant} and Proposition \ref{jjj} that $-4\lambda\dim(\g)=\tr(\mathcal{J}_1)+\tr(\hat{\mathcal{B}})$, furthermore $\g$ is completely solvable and so its Killing form is positive definite, in particular, using that $\hat{\mathcal{B}}(e)=\hat{\mathcal{B}}(f_i)=0$, we get that :
	$$\tr(\hat{\mathcal{B}})=\sum_{i=1}^s \langle\hat{\B}(g_i),g_i\rangle\geq 0.$$
	On the other hand, it is easy to see that for $i=1,\dots,d$ :
	$$\tr(S_i^2)=2\alpha_i^2-\underset{u_i}{\underbrace{\sum_{j,k=1}^d\langle S_i(f_j),f_k\rangle^2-\sum_{j,k=1}^d\langle S_i(g_j),g_k\rangle^2-2\sum_{j=1}^d\sum_{k=1}^s\langle S_i(f_j),g_k\rangle^2}}.$$
	Hence the previous equation shows that
	$$-4\lambda(\dim\g-1)=2\tr(\hat{\B})+\sum_{i=1}^d u_i,$$
	which is only possible if $\lambda=0$, $\tr(\hat{B})=0$ and $u_i=0$ for all $i=1,\dots,d$, this leads to $S_i(e)=0$ and $S_i((\R e)^\perp)\subset\R e$, which means that for all $j,k=1,\dots,d$ and all $p,\ell=1,\dots, s$ :
	\begin{equation}
		\label{eqbrs}
		[f_j,f_k]=\langle K(f_j),f_k\rangle e,\quad [g_p,g_\ell]=\langle K(g_p),g_\ell\rangle e\quad\text{and}\quad [f_j,g_\ell]=\langle K(f_j),g_\ell\rangle e.
	\end{equation} 
	Set $\g_0:=\vect(f_1,\dots,f_d,g_1,\dots,g_s)$ and $I:=(\R e)^\perp$. Then $I$ is an ideal of $\g$ and $I=\R e\oplus\g_0$, furthermore $\R e$ is an ideal of $I$, this follows from the fact that $S_i(e)=0$ for $i=1,\dots,d$, thus the bilinear map $\br_0:\g_0\times\g_0\longrightarrow\g_0$ given by
	$$ [u,v]:=[u,v]_0+\langle K(u),v\rangle e\quad u,v\in\g_0,$$
	defines a Lie bracket on $\g_0$ and from \eqref{eqbrs} we deduce that $(\g_0,\br_0)$ is an abelian Lie algebra. Finally since $\g$ is unimodular we obtain that for any $u\in\g_0$ :
	$$ 0=\tr(\ad_u)=\langle [u,e],\overline{e}\rangle+\langle [u,\overline{e}],e\rangle+\sum_{i=1}^d\underset{=0}{\underbrace{\langle [u,f_i],f_i\rangle}}=\langle K(u),e\rangle +\sum_{i=1}^d\langle S_i(u),\overline{e}\rangle\underset{=0}{\underbrace{\langle f_i,e\rangle}}=\langle K(u),e\rangle.$$
	This means that $[e,u]=0$ for all $u\in\g_0$. In summary $\g=\R e\oplus\g_0\oplus\R\overline{e}$ and the only non-vanishing brackets of $\g$ are of the form :
	$$[\overline{e},e]=\mu e,\quad[\overline{e},u]=D(u)+\langle b,u\rangle e\quad\text{and}\quad [u,v]=\langle K(u),v\rangle e,\quad u,v\in\g_0,$$
	where $\mu:=\langle K(\overline{e}),e\rangle$, $D(u):=\sum_i\langle S_i(\overline{e}),u\rangle f_i$ and $b:=K(\overline{e})$. In other words $\g$ is obtained by a double extension process from an abelian Lie algebra $\g_0$ with parameters $(K,D,\mu,b)$ and since $\g$ is unimodular then $\mu=-\tr(D)$.
\end{proof}
\begin{theo}\label{thdext2}
	Let $\g$ be a completely solvable unimodular, Lorentzian Lie algebra such that $Z(\g)$ is non-trivial and degenerate. Then $\g$ is Einstein if and only if it is obtained by a double extension process from an abelian Lie algebra with admissible parameters $(K,D,0,b)$. In particular $\g$ is Ricci-flat.
\end{theo}
\begin{proof}
	Let $e\in\mathrm{Z}(\g)$ be an istropic vector and choose $\overline{e}\in\g$ and $f_1,\dots f_d\in\g$ such that $\mathbb{B}:=(e,\overline{e},f_1,\dots,f_d)$ is a Lorentzian basis of $\g$ then denote $(K,\overline{K},S_1,\dots,S_d)$ the structure endomorphisms of $\g$ corresponding to $\mathbb{B}$. Since $e\in\mathrm{Z}(\g)$ we get that $\mathcal{J}_1(e)=\hat{\B}(e)=0$ and thus $4\lambda e=\mathcal{J}_2(e)$. On the other hand, we have for all $u\in\g$ :
	\begin{eqnarray*}
		-\mathcal{J}_2(u)=\langle u,e\rangle\tr(K^2)e+\langle u,e\rangle\tr(K\circ\overline{K})\overline{e}+\sum_{i=1}^d\langle u,e\rangle\tr(K\circ S_i)f_i+\langle u,\overline{e}\rangle\tr(\overline{K}^2)\overline{e}+\langle u,\overline{e}\rangle\tr(\overline{K}\circ K)e \\+\sum_{i=1}^d\langle u,\overline{e}\rangle\tr(\overline{K}\circ S_i)f_i+\sum_{i=1}^d\langle u,f_i\rangle\tr(S_i\circ K)e+\sum_{i=1}^d\langle u,f_i\rangle\tr(S_i\circ\overline{K})\overline{e}+\sum_{i,j=1}^d\langle u,f_i\rangle\tr(S_i\circ S_j)f_j.
	\end{eqnarray*}
	It follows that $0=\langle 4\lambda e,e\rangle=\langle\mathcal{J}_2(e),e\rangle=\tr(\overline{K}^2)$ and since $\overline{K}(e)=0$ then Lemma \ref{le} implies that $\overline{K}((\R e)^\perp)\subset \R e$ and $\tr(\overline{K}\circ K)=\tr(\overline{K}\circ S_i)=0$, in particular 
	$$\lambda=\langle\lambda e,\overline{e}\rangle=\dfrac{1}{4}\langle\mathcal{J}_2(e),\overline{e}\rangle=\tr(\overline{K}\circ K)=0,$$
	i.e $\g$ is Ricci-flat. Next we have that $\g$ is completely solvable which means that its Killing form is positive semi-definite, hence using the formula
	$$-\mathcal{J}_1=K\circ\overline{K}+\overline{K}\circ K+\sum_{i=1}^d S_i^2,$$
	and the fact that $\hat{\B}(e)=0$ we obtain that  $$\tr(\mathcal{J}_1)=-\sum_{i=1}^d\tr(S_i^2)=\sum_{i,j,k=1}^d\langle S_i(f_j),f_k\rangle^2\geq 0\quad\text{and}\quad \tr(\hat{\mathcal{B}})=\sum_{i=1}^d\langle \hat{\B}(f_i),f_i\rangle\geq 0,$$
	thus by taking the trace of the Ricci operator and using that $\g$ is Ricci-flat we get that :
	$$0=-\dfrac{1}{4}\tr(\mathcal{J}_1)-\dfrac{1}{2}\tr(\hat{\mathcal{B}}),$$
	which shows in particular that $\langle S_i(f_j),f_k\rangle=0$ for all $i,j,k=1,\dots,d$ and $S_i((\R e)^\perp)\subset \R e$, i.e
	\begin{equation}
		\label{eqbrs2}
		[f_i,f_j]=\langle K(f_i),f_j\rangle e,\quad i,j=1,\dots,d.
	\end{equation}
	Now set $\g_0:=\vect(f_1,\dots,f_d)$, since $e\in\mathrm{Z}(\g)$ we get that the bilinear map $\br_0:\g_0\times\g_0\longrightarrow\g_0$ given by
	$$ [u,v]:=[u,v]_0+\langle K(u),v\rangle\quad u,v\in\g_0,$$
	defines a Lie bracket on $\g_0$ and from \eqref{eqbrs2} we deduce that $(\g_0,\br_0)$ is an abelian Lie algebra. Finally using that $\g$ is unimodular we obtain that for any $u\in\g_0$ :
	$$0=\tr(\ad_u)=\langle [u,\overline{e}],e\rangle+\sum_{i=1}^d\underset{=0}{\underbrace{\langle [u,f_i],f_i\rangle}}=\langle\overline{K}(u),\overline{e}\rangle=-\langle\overline{K}(\overline{e}),u\rangle.$$
	Therefore $\overline{K}(\overline{e})\in\R e$ and the only non-vanishing Lie brackets of $\g_0$ are of the form :
	$$ [\overline{e},u]=D(u)+\langle b,u\rangle e,\quad [u,v]=\langle K(u),v\rangle e,\quad\text{for}\;u,v\in\g,$$
	where $D(u)=\sum_i\langle S_i(\overline{e}),u\rangle f_i$ and $b=K(\overline{e})$. In other words $\g$ is obtained by a double extension process from an abelian Lie algebra $\g_0$ with parameters $(K,D,0,b)$ and $\tr(D)=0$. 
\end{proof}


\begin{thebibliography}{99}
	
	\bibitem{bou1} M. Ait Haddou, M. Boucetta, H. Lebzioui,  Left-invariant Lorentzian flat metrics on Lie groups,
	Journal of Lie Theory 22 (2012), No. 1, 269-289.
	
	
	
	
	\bibitem{Aub-Med}  \textsc{ Aubert Anne \& Medina Alberto,} Groupes de Lie
	pseudo-riemanniens plats.  Tohoku Math. J. {\bf (2) 55} (2003), no. 4, 487-506.
	
	\bibitem{besse} Arthur L. Besse, Einstein manifolds, Classic in Mathematics Springer (2008).
	
	\bibitem{bou} Mohamed Boucetta, Hicham Lebzioui, Non unimodular Lorentzian flat Lie algebras, To appear in Communication in Algebra.
	
	\bibitem{bouc1}
	Boucetta, M., Ricci flat left invariant Lorentzian metrics on
	2-step nilpotent Lie groups, arXiv:0910.2563v2[math.DG] 15 Feb 2010.
	
	\bibitem{bou0} M. Boucetta, O. Tibssirte, On Einstein Lorentzian nilpotent Lie groups, Journal of Pure and Applied Algebra.
	\bibitem{Boucetta2}
	M. Boucetta, O. Tibssirte. Classification of Einstein Lorentzian 3-nilpotent Lie groups with 1-dimensional nondegenerate center. Journal of Lie Theory.
	
	\bibitem{calvaruso3} Giovanni Calvaruso and Amirhesam Zaeim, Four-dimensional homogeneous Lorentzian manifolds, Monatsh Math (2014) 174:377-402.
	
	\bibitem{dotti} \textsc{Dotti Isabel}, Ricci curvature of left invariant
	metrics on solvable unimodular Lie groups, Math. Z. {\bf180} (1982) 257-263.
	
	\bibitem{guediri} Mohammed Guediri \& Mona Bin-Asfour,
	Ricci-flat left-invariant Lorentzian metrics on 2-step nilpotent Lie groups,
	Archivum Mathematicum, Vol. 50 (2014), No. 3, 171-192.
	
	\bibitem{heber} Heber J., Noncompact Einstein spaces, Invent. Math. {\bf 133} (1998) 279-352.
	
	\bibitem{goze} Goze M. \&  Khakimdjanov Y., Nilpotent Lie algebras, Mathematics and Its Applications
	Vol. 361 (1996) SPRINGER-SCIENCE-BUSINESS MEDIA, B.V.
	
	\bibitem{lauret} J. Lauret, A canonical compatible metric for geometric structures on nilmanifolds, Ann. Global Anal. Geom. {\bf 30} (2006) 107-138.
	
	\bibitem{lauret1} J. Lauret, Einstein solvmanifolds are standard, Annals of Mathematics, {\bf 172} (2010), 1859-1877.
	
	
	
	\bibitem{medina} Alberto Medina \& Philippe Revoy, Alg\`ebre de Lie et produit scalaire invariant,
	Annales scientifiques de l'\'Ecole Normale Sup\'erieure (1985)
	Volume: 18, Issue: 3, page 553-561.
	
	
	\bibitem{milnor} \textsc{ Milnor J.},  Curvature of left invariant metrics on Lie
	groups, Adv. in Math. {\bf 21} (1976), 283-329.
	
	\bibitem{petrov} A. Z. Petrov, Einstein spaces, Pergamon Press (1969).
	
	\bibitem{dotti} Miatello, Isabel Dotti. "Ricci curvature of left invariant metrics on solvable unimodular Lie groups." Mathematische Zeitschrift 180.3 (1982): 257-263.
	
	\bibitem{Dconti1} Conti, D., Rossi, F. A. (2019). Einstein nilpotent Lie groups. Journal of Pure and Applied Algebra, {\bf 223(3)}, 976-997.
	
\end{thebibliography}
\end{document}